\documentclass[12pt]{amsart}

\def\tfp<#1,#2>{\ll\hskip-4pt#1 , #2\hskip-4pt\gg}
\def\tp<#1,#2,#3>{\langle#1 , #2,#3\rangle}

\def\bx{x}\def\by{y}
\def\xa{\hat{x}}\def\ya{\hat{y}}\def\wa{\hat{w}}
\def\xb{x_0}\def\yb{y_0}
\def\xc{\hat{x}_0}\def\yc{\hat{y}_0}\def\wc{\hat{w}_0}
\def\Q{\mathbb Q}
\def\D{\det A}

\theoremstyle{plain}
\newtheorem*{Theorem*}{Theorem}
\theoremstyle{definition}
\newtheorem*{Definition*}{Remark}

\begin{document}
\title{The Heisenberg 3-manifold is canonically oriented}
\author{Laurence R.~Taylor}
\address{Department of Mathematics
 \newline\indent
University of Notre Dame
Notre Dame, IN 46556, USA}
\email{taylor.2@nd.edu}
\begin{abstract}
This note describes a canonical way to orient the Heisenberg $3$-manifold.
\end{abstract}
\maketitle
The Heisenberg manifold $M^3$ is an example of a nil manifold 
with a minimal model consisting
of classes $\xa$, $\ya$, $\wa$ in degree one and 
$d \wa = \xa\wedge \ya$. 
The classes $\xa$ and $\ya$ are a basis for the $1$-cocycles. 
If $\bx$ and $\by$ are the corresponding cohomology classes 
they are a basis for $H^1$. 
The cohomology class of the cocycle $\xa\wedge\ya\wedge\wa$ 
is a generator of $H^3(M;\Q)\cong\Q$. 
An orientation of $M$ also determines a generator $[M]^\ast\in H^3(M;\Q)$. 
The \emph{canonical orientation} of $M$ is the orientation 
so that $\xa\wedge\ya\wedge\wa = r\, [M]^\ast$ with $r>0$. 
The theorem below shows the orientation is canonical. 

Define a pairing
$H^1\times H^1 \to H^3$ as follows. 
Let $\xb$ and $\yb$ be any two elements in $H^1$. 
Choose 1-cocycles $\xc$ and $\yc$ whose cohomology classes
are $\xb$ and $\yb$ respectively. 
Since $\xb\cup\yb=0\in H^2$, there exists $\wc$ such that 
$d\wc=\xb\wedge\yb$.
Define $\tfp<\xb,\yb>$ to be the cohomology class of $\xc\wedge\yc\wedge\wc$.
The next theorem shows this pairing is well-defined. 
\begin{Theorem*}
For any two elements $\xb$, $\yb$ in $H^1$, 
let $A$ be the matrix expressing these classes in terms of the basis $\bx$, $\by$.  
Then $\tfp<\xb,\yb>\, =(\det A)^2\tfp<\bx,\by>$. 
Hence the pairing is well-defined and the canonical orientation can be determined 
from any basis for $H^1(M;\Q)$.
\end{Theorem*}
\begin{proof}
Let $\xb = a\bx + b\by$ and $\yb=c\bx+d\by$ so 
$\det A=a d - b c$. 
Since the map from the 1-cocycles to $H^1$ is an isomorphism, 
the only choices for 1-cocycles are $\xc = a\xa + b\ya$ and $\yc=c\xa+d\ya$. 
Then $\xc\wedge \yc = (\D)\xa\wedge\ya$. 
Any choice for 
$\wc$ must be of the form $(\D)\wa + p\xa+q\ya$. 
Check 
$\xc\wedge\yc\wedge(p\xa+q\ya)=0$. 
Then $\xc\wedge\yc\wedge\wc=(\D)^2\xa\wedge\ya\wedge\wa$.
\end{proof}
\begin{Definition*}
The pairing $\tfp<\xb,\yb>$ is equal to
$\frac12\xb\cup\tp<\yb,\xb,\yb>$ where $\tp<\yb,\xb,\yb>$ is the Massey triple product. 
Hence the canonical orientation can also be determined using singular cohomology or de~Rahm cohomology. 
\end{Definition*}
\end{document}